\documentclass[12pt]{amsart}
\usepackage{graphicx,verbatim,amssymb}
\usepackage[active]{srcltx}

\newtheorem{thm}{Theorem}[section]
\newtheorem{cor}[thm]{Corollary}
\newtheorem{lem}[thm]{Lemma}
\newtheorem{prop}[thm]{Proposition}

\newtheorem*{thmA}{Theorem A}
\newtheorem*{thmB}{Theorem B}

\theoremstyle{definition}
\newtheorem{defn}[thm]{Definition}
\theoremstyle{remark}
\newtheorem{rem}[thm]{Remark}
\numberwithin{equation}{section}

\def\R{\mathbb{R}}
\def\Z{\mathbb{Z}}
\def\C{\mathbb{C}}
\def\Q{\mathbb{Q}}
\def\F{\mathbb{F}}
\def\xb{\mathbf{x}}
\def\yb{\mathbf{y}}
\def\zb{\mathbf{z}}
\def\ab{\mathbf{a}}  \def\bb{\mathbf{b}}
\def\cb{\mathbf{c}}  \def\db{\mathbf{d}}
\def\eb{\mathbf{e}}
  
\newcommand{\eps}{\varepsilon}

\newcommand{\A}{\mathcal{A}}

\def\0{\varnothing}

\def\ol{\overline}

\begin{document}

\title[Disquisitiones 235]
{Disquisitiones 235}
\author{Vladlen Timorin}

\begin{abstract}
  Section 235 of Gauss' fundamental treatise ``Disquisitiones Arithmeticae''
establishes
basic properties that compositions of binary quadratic forms must satisfy.
Although this section is very technical, it contains truly important results.
We review section 235 using a more invariant language and simplifying the arguments.
We also make the statements slightly stronger by removing unnecessary
assumptions.
\end{abstract}

\address[Vladlen Timorin]{Faculty of Mathematics and Laboratory of Algebraic Geometry,
National Research University Higher School of Economics,
7 Vavilova St 117312 Moscow; Independent University of Moscow,
Bolshoy Vlasyevskiy Pereulok 11, 119002 Moscow, Russia}

\email{vtimorin@hse.ru}

\thanks{
The author was partially supported by
the Dynasty Foundation grant, RFBR grants
11-01-00654-a, 12-01-33020, 13-01-12449,
and AG Laboratory NRU HSE, MESRF grant ag. 11 11.G34.31.0023
}

\maketitle

\section{Introduction}

One of the major motivations of Gauss' ``Disquisitiones'' \cite{G} was the following question:
what are all possible values that a given binary integer quadratic form can take?
A \emph{binary quadratic form} in variables $x$ and $y$ is a function
$$
f(x,y)=ax^2+bxy+cy^2,
$$
where $a$, $b$, $c$ are given numbers.
We will use the abbreviation $(a,b,c)$ for the quadratic form as above.
In the sequel, by a \emph{form} we always mean a binary quadratic form.
A form $(a,b,c)$ is \emph{integer} if
$a$, $b$, $c\in\Z$ (we write $\Z$ for the set of all integers).\footnote{
In fact, Gauss only considered forms $(a,b,c)$, for which $b$ is even,
and used $(a,\frac b2,c)$ to denote these forms.
}
An integer $m$ is said to be \emph{representable} by a quadratic form $f$
if $m=f(x,y)$ for some $x$, $y\in\Z$.

The following observations lead to the notion of composition of quadratic forms:
\begin{enumerate}
  \item The product of two integers representable as sums of squares, i.e.
  representable by $(1,0,1)$, is also representable as a sum of squares.
  This follows from the formula
  $$
  (x^2+y^2)({x'}^2+{y'}^2)=(xx'-yy')^2+(xy'+x'y)^2.
  $$

  \item More generally, the product of two integers representable by
  $(1,0,d)$, where $d\in\Z$, is also representable by $(1,0,d)$.
  This is explained by the formula
  $$
  (x^2+dy^2)({x'}^2+d{y'}^2)=(xx'-dyy')^2+d(xy'+x'y)^2.
  $$
  In \cite{Ai,AT}, a large class of binary quadratic forms is
  described that satisfy the \emph{semigroup law}: the product of two
  integers representable by any form of this class is also representable
  by the same form. See \cite{EF} for a complete description of forms
  obeying the semigroup law.

  \item It often happens that the product of two integers $m$ and $m' $representable
  by different quadratic forms $f$ and $f'$ is representable by a third
  quadratic form $F$; moreover, the form $F$ depends only on $f$ and $f'$ but
  not on $m$ and $m'$. For example, the product of two integers representable
  by $(2,2,3)$ is an integer representable by $(1,0,5)$, as follows from the formula
$$
(2x^2+2xy+3y^2)(2{x'}^2+2{x'}y'+3{y'}^2)=X^2+5Y^2,
$$
where $X=2xx'+xy'+yx'-2yy'$ and $Y=xy'+yx'+yy'$.

  \item Although the product of \emph{two} integers representable by a
  given quadratic form $f$ is not always representable by $f$, the product
  of \emph{three} integers representable by $f$ is always representable by $f$.
  This is what Vladimir Arnold named the \emph{trigroup law} \cite{Ar}.
\end{enumerate}

Gauss did not use vector notation.
However, it helps to see an invariant meaning behind the statements of Gauss.
We will use bold letters to denote 2-dimensional vectors.
For example, if $\xb=(x,y)$ and $f$ is a quadratic form, then $f(\xb)$ is
the same as $f(x,y)$.
Thus, we can think of $f$ as a function on $\Z^2$.
We can even replace $\Z^2$ with any 2-dimensional lattice $L_f$, i.e. a free
abelian $\Z$-module of rank 2.
The notion of a lattice often includes a quadratic form $f:L_f\to\R$
(we use $\R$ to denote the set of all real numbers).
A geometric viewpoint on lattices and quadratic forms consists of considering
$f(\xb)$ as the square of the length of $\xb$, i.e. the inner (dot) product of $\xb$
with itself.
Note, however, that $f(\xb)$ does not need to be positive.
A form $f$ defines its \emph{polarization}, which is an analog of the dot product.
The polarization of $f$ is a symmetric bilinear functional $f(\xb,\yb)$
of $\xb$ and $\yb\in L_f$ such that $f(\xb,\xb)=f(\xb)$.
Abusing the notation, we use the same letter $f$ to denote a quadratic form and
its polarization: if $x$ and $y$ are numbers, then $f(x,y)$ is the value of $f$
at $(x,y)$; on the other hand, if $\xb$ and $\yb$ are vectors, then $f(\xb,\yb)$
is the value of the polarization of $f$ at $\xb$ and $\yb$.
We will also call $f(\xb,\yb)$ the \emph{inner product} of $\xb$ and $\yb$
(with respect to $f$).

\begin{rem}[Integrality]
  Using a more invariant language, we can say that a form $f$ on $L_f$ is
  integer if the values $f(\xb)$ are integers for all $\xb\in L_f$.
  Note that there is a competing, more restrictive, version of integrality: we may want to
  require that all inner products $f(\xb,\yb)$ be integer.
  This second version of integrality is equivalent (in case $f=(a,b,c)$)
  to saying that the numbers $a$, $\frac b2$, $c$ are integers.
  Although Gauss used the latter version, we will use the former, more
  general, version of integrality.
\end{rem}

\begin{defn}[Composition]
Consider a bilinear map $\circ:\Z^2\times\Z^2\to\Z^2$ (the value of $\circ$
at vectors $\xb$ and $\xb'$ is denoted by $\xb\circ\xb'$).
Suppose that $f$, $f'$ and $F$ are three quadratic forms on $\Z^2$ satisfying the relation
\begin{equation}
\label{e:comp}
F(\xb\circ\xb')=f(\xb)f'(\xb'),\quad\forall\ \xb,\xb'\in\Z^2.
\end{equation}
Suppose additionally that the vectors $\xb\circ\xb'$ span $\Z^2$.
Then $\circ$ is called a \emph{composition law}, and the form $F$
is said to be \emph{composed} of $f$ and $f'$.
\end{defn}

\begin{rem}[Lattices in quadratic rings]
\label{r:rings}
The notion of composition is motivated by the examples listed above.
There is also a very important class of examples, which led to modern
algebraic treatment of class groups.
Namely, suppose that $R$ is a \emph{quadratic ring}, i.e. a ring,
whose underlying additive group is isomorphic to $\Z^2$ (e.g. we
may consider the ring $\Z\oplus\sqrt{-1}\Z$ of Gaussian integers or
the ring of all algebraic integers contained in a quadratic field
$\Q(\sqrt{d})$, the field obtained from the field $\Q$ of all rational
numbers by adjoining a square root of some square-free integer $d$).
Suppose that $N$ is a quadratic form on $R$ (usually called a \emph{norm})
such that $N(uv)=N(u)N(v)$ for all $u$, $v\in R$.
For example, if $R$ is the ring of Gaussian integers, we may set
$N(u)=u\ol u$ to be the square of the modulus.
Consider sublattices $L$, $L'\subset R$, and let $f$, $f'$ be
the restrictions of $N$ to $L$, $L'$ (respectively).
Define the lattice $LL'$ as the additive subgroup of $R$ generated by
all products $uu'$, where $u\in L$ and $u'\in L'$.
Finally, define the quadratic form $F$ on $LL'$ as the restriction of $N$ to $LL'$.
Then the multiplication in $R$ defines a composition law, as we always have
$$
F(uu')=f(u)f'(u').
$$
F. Klein \cite{K} argues that Gauss had been aware of these examples but opted not
to use more sophisticated objects than just integers, not even complex numbers.
\end{rem}

We now introduce some basic invariants of binary quadratic forms that were
studied by Gauss in his ``Disquisitiones''. A more systematic
treatment of these invariants has been given later by H. Minkowski \cite{M}.
For every integer quadratic form $f:L_f\to\Z$, we let $\delta(f)$ denote
the greatest common divisor of all values $f(\xb)$, where $\xb$ runs through $L_f$.
We let $\delta'(f)$ denote the greatest common divisor of all doubled inner products
$2f(\xb,\yb)$, where $\xb$ and $\yb$ run through $L_f$.
If $f=(a,b,c)$, then we have
$$
\delta(f)=\gcd(a,b,c),\quad \delta'(f)=\gcd(2a,b,2c).
$$
Clearly, the number $\delta(f)$ divides the number $\delta'(f)$,
and the quotient $\sigma(f)=\delta'(f)/\delta(f)$ is equal to 1 or 2.

We need to explain the \emph{invariance} property
of the invariants $\delta(f)$ and $\sigma(f)$.
Note that the group $GL_2(\Z)$ of all automorphisms of the lattice $\Z^2$
(the automorphisms of $\Z^2$ can be represented by integer matrices
of determinant $\pm 1$)
acts on quadratic forms: an automorphism $A\in GL_2(\Z)$ takes a form
$f$ to the form $f':\xb\mapsto f(A\xb)$.
The new form $f'$ takes exactly the same values at points of $\Z^2$ as
the old form $f$.
Thus, if we study numbers representable by quadratic forms, we may not
distinguish between the forms $f$ and $f'$.
Such forms (that are obtained from each other by an automorphism $A\in GL_2(\Z)$)
are said to be \emph{equivalent}.
If, additionally, $\det(A)=1$, i.e. if $A$ preserves the orientation, then
we say that $f$ and $f'$ are \emph{properly equivalent}.
The invariance of, say, $\delta(f)$ means that $\delta(f)=\delta(f')$ for
any pair of equivalent forms $f$, $f'$.
This is clear since $\delta(f)$ depends only on the set of integers representable
by $f$, and this set does not change under the action of $GL_2(\Z)$.
Similarly, it is easy to see that $\sigma(f)$ is also an invariant.

\begin{defn}[Discriminant]
There is another invariant of a binary form $f$ called the \emph{discriminant} of $f$.
If $f=(a,b,c)$, then the discriminant $d(f)$ of $f$ is defined by the
formula
\begin{equation}
\label{e:det}
d(f)=b^2-4ac.
\end{equation}
The discriminant is also an invariant, although this is not immediately obvious
from the definition. Later on, we will give a more invariant definition of
the discriminant, from which the invariance will be clear.
\end{defn}

It is clear from the definition of the discriminant that, for an integer form $f$,
the number $d(f)$ is divisible by
$\delta(f)^2$. We set $\theta(f)=d(f)/\delta(f)^2$.
Then $\theta(f)$ is also an integer invariant of $f$.
The following theorem of Gauss shows how the invariants $\delta$, $\sigma$ and $\theta$
of integer quadratic forms are related to the corresponding invariants of their composition.

\begin{thmA}
  Suppose that a form $F$ is composed of \emph{integer} forms $f$ and $f'$.
  Then $F$ is also an integer form.
  Moreover, we have
  $$
  \delta(F)=\delta(f)\delta(f'),\quad
  \sigma(F)=\min(\sigma(f),\sigma(f')),\quad
  \theta(F)=\gcd(\theta(f),\theta(f')).
  $$
\end{thmA}

Another result of Section 235 is the following

\begin{thmB}
  A composition of integer quadratic forms $f$ and $f'$ exists if and only
  if the ratio $d(f)/d(f')$ is a square of some rational number.
\end{thmB}

In fact, only the ``only if'' part of this statement is discussed in Section 235.
We give the ``if'' part for completeness; it is discussed in later
sections of ``Disquisitiones''.

\begin{rem}[Perspectives]
  It would be interesting to find multidimensional generalizations of Theorem A.
  It is a joint project of E. Duriev, A. Pakharev and V. Timorin to implement this
  in dimension 4, i.e. for quaternary forms.
  Examples of compositions of quaternary forms can be obtained by multiplying
  certain sublattices in the ring of quaternions.
  Compositions of ternary (3-variable) forms are less interesting because, to a
  large extent, they reduce to compositions of binary forms.
  A multidimensional generalization of Theorem B is given by the classification
  of composition algebras over $\Q$, see \cite{J}.
\end{rem}

\begin{rem}[Ideal classes]
  The language of ideal classes, introduced by R. Dedekind in his
  supplements to \cite{D}, deals with rather
  special type of composition, namely, with composition of forms $f$, $f'$,
  for which $\delta(f)=\delta(f')=1$ and $d(f)=d(f')$.
  In fact, it is exactly Theorem A that allows to reduce the general case
  to the case $\delta(f)=\delta(f')=1$.
  Since ideal classes are defined in any commutative ring, Dedekind's
  approach leads to a nice and general algebraic theory described in many
  modern textbooks.
  On the other hand, the language of ideal classes is not adapted
  for non-commutative generalizations.
\end{rem}

\section{Discriminants}
In this section, we discuss the invariant meaning of discriminants and
prove the ``only if'' part of Theorem B.

For a pair of vectors $\xb=(x,y)$ and $\xb'=(x',y')$, we write
$\det(\xb,\xb')$ for the determinant of the matrix
$$
\begin{pmatrix}
  x&x'\\y&y'
\end{pmatrix}.
$$
It is clear that the expression $\det(\xb,\yb)$ is (almost) invariant under the
action of $GL_2(\Z)$: if $A\in GL_2(\Z)$ is an automorphism of $\Z^2$, then
$$
\det(A\xb,A\yb)=\det(A)\det(\xb,\yb).
$$
Note that $\det(A)=\pm 1$.
In fact, the formula displayed above holds under a more general assumption
that $A:\Z^2\to\Z^2$ is any linear map (i.e. a $\Z$-module homomorphism).
We start with the following general statement from the theory of symmetric polynomials

\begin{lem}
  \label{l:sym}
  Let $Q:\C^2\times\C^2\to\C$ be a bi-quadratic polynomial with the
  following properties:
  $$
  Q(\xb,\yb)=Q(\yb,\xb),\quad Q(\xb,\xb)=0\quad\forall\ \xb,\,\yb\in\C^2.
  $$
  Then there is a number $c$ such that $Q(\xb,\yb)=c\,\det(\xb,\yb)^2$.
\end{lem}

\begin{proof}[Sketch of a proof]
  The lemma can be proved by a straightforward computation.
Alternatively, one can observe that, for every fixed $\yb$, the
polynomial $Q(\xb,\yb)$ (regarded as a quadratic form in $\xb$) vanishes
on the line given by the equation $\det(\xb,\yb)=0$.
One can conclude that $Q(\xb,\yb)$ is divisible by $\det(\xb,\yb)$ in the
ring of polynomials in $\xb$ (whose coefficients may depend on $\yb$).
Since the ratio $Q(\xb,\yb)/\det(\xb,\yb)$ is a skew-symmetric function
of $\xb$ and $\yb$, it depends polynomially (in fact, linearly) both on $\xb$
and on $\yb$. A skew-symmetric bilinear function of $\xb$ and $\yb$ must
have the form $c\,\det(\xb,\yb)$ for some constant $c$.
\end{proof}

The next proposition provides an invariant meaning of the
discriminant of a quadratic form.

\begin{prop}
  \label{p:det}
  Let $f$ be a binary quadratic form.
  Then, for every $\xb$ and $\yb\in\Z^2$, we have
  $$
  4(f(\xb,\yb)^2-f(\xb)f(\yb))=d(f)\det(\xb,\yb)^2.
  $$
\end{prop}

\begin{proof}
Consider the left-hand side as a bi-quadratic function
$Q(\xb,\yb)$ of $\xb$ and $\yb$.
This function satisfies the assumptions of Lemma \ref{l:sym}.
Therefore, we have $Q(\xb,\yb)=d\,\det(\xb,\yb)^2$, where the number $d$
depends only on $f$ but not on $\xb$ or $\yb$.
Substituting $\xb=(1,0)$ and $\yb=(0,1)$, we obtain that
$d=d(f)$, as desired.
\end{proof}

\begin{cor}
  \label{c:det-inv}
  Suppose that $A:\Z^2\to\Z^2$ is a linear map, and quadratic forms
  $f$, $f'$ satisfy the identity $f'(\xb)=f(A\xb)$ for all $\xb\in\Z^2$.
  Then $d(f')=d(f)\det(A)^2$.
  In particular, the discriminant of a quadratic form is an invariant:
  if $f$ and $f'$ are equivalent forms, then $d(f)=d(f')$.
\end{cor}

\begin{proof}
  On the one hand, we have
  $4(f'(\xb,\yb)-f'(\xb)f'(\yb))=d(f')\det(\xb,\yb)^2$
  by Proposition \ref{p:det}.
  On the other hand, the same bi-quadratic form of $\xb$ and $\yb$ is equal to
  $$
  4(f(A\xb,A\yb)^2-f(A\xb)f(A\yb))=d(f)\det(A\xb,A\yb)^2=
  d(f)\det(A)^2\det(\xb,\yb),
  $$
  again by Proposition \ref{p:det}.
  Comparing the right-hand sides, we conclude that $d(f')=d(f)\det(A)^2$,
  as desired.
\end{proof}

We will also need the following proposition essentially due to Lagrange:

\begin{prop}
  \label{p:lagr}
    Consider vectors $\ab$, $\bb$, $\cb$, $\db\in\C^2$ and a quadratic form $F$ on $\C^2$.
  Then we have
  $$
  4\det
  \begin{pmatrix}
    F(\ab,\cb)& F(\ab,\db)\\
    F(\bb,\cb)& F(\bb,\db)
  \end{pmatrix}=-d(F)\det(\ab,\bb)\det(\cb,\db).
  $$
\end{prop}

\begin{proof}
  Note that the left-hand side depends multi-linearly on $\ab$, $\bb$, $\cb$, $\db$.
  Note also that it is skew-symmetric with respect to $\ab$, $\bb$.
  It is also skew-symmetric with respect to $\cb$, $\db$.
  It follows that the left-hand side is equal to $\det(\ab,\bb)\det(\cb,\db)$
  times some constant that depends only on $F$.
  This constant can be computed by setting $\ab=\cb=(1,0)$ and
  $\bb=\db=(0,1)$.
  If $F=(A,B,C)$, then the determinant in the left-hand side is equal to
  $$
  4\det
  \begin{pmatrix}
    A& B/2\\B/2& C
  \end{pmatrix}=-d(F).
  $$
  The product $\det(\ab,\bb)\det(\cb,\db)$ is equal to 1.
\end{proof}

In the rest of this paper, we fix integer forms $f$, $f'$, and a form $F$
composed of $f$ and $f'$ with the help of a composition law $\circ$.
Recall formula (\ref{e:comp}):
$$
F(\xb\circ\xb')=f(\xb)f'(\xb').\leqno{(\ref{e:comp})}
$$
Let us fix $\xb'$ and consider both sides of this formula as quadratic forms in $\xb$.
Compute the discriminants of both sides.
The discriminant of the left-hand side is equal to the discriminant of $F$
times the square of the determinant of the linear map $\circ\xb'$ that
takes $\xb$ to $\xb\circ\xb'$.
To compute the discriminant of the right-hand side, it suffices to observe
that $d(\lambda f)=\lambda^2 d(f)$ for every scalar factor $\lambda$.
Thus we have
\begin{equation}
\label{e:circy}
  d(F)\det(\circ\xb')^2=d(f) f'(\xb')^2.
\end{equation}

Similarly, if we fix $\xb$ and consider both sides of equation (\ref{e:comp})
as quadratic forms in $\xb'$, then we obtain the equality
\begin{equation}
\label{e:xcirc}
  d(F)\det(\xb\circ)^2=d(f')f(\xb)^2,
\end{equation}
where the linear map $\xb\circ:\Z^2\to\Z^2$ takes $\xb'$ to $\xb\circ\xb'$.
Equations (\ref{e:circy}) and (\ref{e:xcirc}) imply the following proposition.

\begin{prop}
  \label{p:onlyif}
  Suppose that a quadratic form $F$ is composed of integer quadratic forms
  $f$ and $f'$ by means of a composition law $\circ$.
  Then
  \begin{enumerate}
    \item the determinant $d(F)$ is a rational number;
    \item the ratio $d(f)/d(f')$ is a square of a rational number;
    \item the quadratic forms $f(\xb)$, $f'(\xb')$ are proportional, respectively,
    to $\det(\xb\circ)$ and $\det(\circ\xb')$:
    $$
    \det(\xb\circ)=\nu f(\xb),\quad \det(\circ\xb')=\nu' f'(\xb').
    $$
    \item the three determinants $d(F)$, $d(f)$, $d(f')$ are related as follows:
    $$
    d(f')=\nu^2 d(F),\quad d(f)={\nu'}^2 d(F).
    $$
  \end{enumerate}
\end{prop}

In particular, this proves the ``only if'' part of Theorem B.

\begin{rem}[Direct composition and uniqueness]
  The composition $\circ$ is said to be \emph{direct} if both $\nu$ and $\nu'$ are positive.
  Different direct composition laws applied to the same pair of integer quadratic forms $f$, $f'$
  give rise to equivalent forms $F$.
  Moreover, the composition laws themselves are equivalent in a natural sense.
  A proof of this statement and some generalizations of it can be found in \cite{DB}.
  Unfortunately, an analog of this statement fails in higher dimensions, e.g., in
  dimension 4.
\end{rem}

\section{Integrality}
Recall that we fixed integer forms $f$, $f'$, a form $F$ composed
of $f$, $f'$ (a priori not integer), and the corresponding composition law $\circ$.
In this section, we prove Theorem A.

Identity (\ref{e:comp}) can be rewritten as follows
$$
F(\xb\circ\yb,\xb\circ\yb)=f(\xb,\xb)f'(\yb,\yb).
$$
Note that both sides of this identity are bi-quadratic
functions of $\xb$ and $\yb$.
The corresponding identity on the associated bi-symmetric 4-linear functions looks as follows:
\begin{equation}
\label{e:mix}
F(\xb\circ\yb,\xb'\circ\yb')+F(\xb\circ\yb',\xb'\circ\yb)=
2f(\xb,\xb')f'(\yb,\yb').
\end{equation}
To see how equation (\ref{e:mix}) follows from (\ref{e:comp}), observe that
both sides of (\ref{e:mix}) are symmetric bilinear functions of $\xb$, $\xb'$
and symmetric bilinear functions of $\yb$, $\yb'$.
Two such functions are equal if and only if they are equal for $\xb=\xb'$
and $\yb=\yb'$.

We now set $\yb=\xb$ and $\yb'=\xb'$.
Making these substitutions into (\ref{e:mix}), we obtain
\begin{equation}
\label{e:mix1}
F(\xb\circ\xb,\xb'\circ\xb')+F(\xb\circ\xb',\xb'\circ\xb)=
2f(\xb,\xb')f'(\xb,\xb').
\end{equation}

Consider the following bi-quadratic form
\begin{equation}
\label{e:Q}
Q(\xb,\yb)=\frac 12\left(F(\xb\circ\xb,\yb\circ\yb)-F(\xb\circ\yb,\yb\circ\xb)
\right).
\end{equation}
Observe that $Q$ satisfies the assumption of the Lemma \ref{l:sym}.
Therefore, we have $Q(\xb,\yb)=\Delta\det(\xb,\yb)^2$ for some constant
$\Delta$ that only depends on $F$, $f$, $f'$ and $\circ$.
In this section, we find an expression for the number $\Delta$.
Using equation (\ref{e:mix1}), we can give two
formulas involving $\Delta$:
\begin{equation}
\label{e:Del1}
F(\xb\circ\xb,\yb\circ\yb)=f(\xb,\yb)f'(\xb,\yb)+\Delta\det(\xb,\yb)^2.
\end{equation}
\begin{equation}
\label{e:Del2}
F(\xb\circ\yb,\yb\circ\xb)=f(\xb,\yb)f'(\xb,\yb)-\Delta\det(\xb,\yb)^2.
\end{equation}

Let us now introduce some short-hand notation that will be used in the
rest of the paper.
We will write $D$, $d$ and $d'$ for the discriminants of $F$, $f$ and $f'$,
respectively.
We will also assume that $f=(a,b,c)$, $f'=(a',b',c')$ and $F=(A,B,C)$.
The numbers $a$, $b$, $c$, $a'$, $b'$, $c'$ are integer.

\begin{prop}
  \label{p:Del}
  The number $4\Delta$ is equal to $D\nu\nu'$, hence $16\Delta^2=dd'$.
  In particular, the number $4\Delta$ is integer.
\end{prop}

Note that the numbers $\nu$, $\nu'$ are a priori only rational, not necessarily
integer.
However, the rational number $D\nu\nu'$ is a square root of the integer $dd'$,
hence it is also an integer.

\begin{proof}
In Proposition \ref{p:lagr}, we set
$\ab=\cb=\xb\circ\xb$, $\bb=\xb\circ\yb$, $\db=\yb\circ\xb$.
Then the left-hand side of the formula is
\begin{align*}
4(F(\xb\circ\xb,\xb\circ\xb)F(\xb\circ\yb,\yb\circ\xb)-
F(\xb\circ\xb,\yb\circ\xb)F(\xb\circ\yb,\xb\circ\xb))=\\
=4f(\xb)f'(\xb)(f(\xb,\yb)f'(\xb,\yb)-\Delta\det(\xb,\yb)^2)-
4f(\xb,\yb)f'(\xb)f(\xb)f'(\xb,\yb)=\\
=-4f(\xb)f'(\xb)\Delta\det(\xb,\yb)^2.
\end{align*}
The right-hand side of the formula is equal to
$-D\det(\xb\circ)\det(\circ\xb)\det(\xb,\yb)^2$.
Thus we obtain that $4f(\xb)f'(\xb)\Delta=D\det(\xb\circ)\det(\circ\xb)$.
From Proposition \ref{p:onlyif}, part (3), we conclude that
$$
4\Delta=D\nu\nu',
$$
as desired.
The formula $16\Delta^2=dd'$ now follows from part (4) of Proposition \ref{p:onlyif}.
\end{proof}

\begin{prop}
  \label{p:int}
  The quadratic form $F$ is integer.
\end{prop}

\begin{proof}
It follows from formula (\ref{e:comp}) that the values of $F$ at all
vectors of the form $\xb\circ\yb$ are integers.
Since the ``products'' $\xb\circ\yb$ generate the lattice $\Z^2$, it suffices
to prove that the numbers $2F(\xb\circ\yb,\xb'\circ\yb')$ are integer for
any vectors $\xb$, $\yb$, $\xb'$, $\yb'\in\Z^2$.
We may assume that each of the vectors $\xb$, $\yb$, $\xb'$, $\yb'$
is equal to $\eb_1=(1,0)$ or $\eb_2=(0,1)$.
Thus, among these four vectors, only two are different
(if all four vectors are equal, then the statement is obvious).
Set $\eb_{ij}=\eb_i\circ\eb_j$.
There are 6 inner products we need to check, namely,
$2F(\eb_{ij},\eb_{i',j'})$ for $(i,j)\ne (i',j')$, and verify that they
all are integer.

The four numbers
$$
2F(\eb_{i1},\eb_{i2})=f(\eb_i)b',\quad 2F(\eb_{1j},\eb_{2j})=bf'(\eb_j),\quad i,j=1,2
$$
are integer.
The remaining two numbers are equal to
$$
2F(\eb_{12},\eb_{21})=\frac 12bb'-2\Delta,\quad 2F(\eb_{11},\eb_{22})=\frac 12bb'+2\Delta
$$
by formulas (\ref{e:Del1}) and (\ref{e:Del2}).
Note that, if both $d$, $d'$ are odd (equivalently, both $b$, $b'$ are odd), then
both $\frac 12bb'$ and $2\Delta$ are half-integers.
If at least one of the numbers $d$, $d'$ is even (equivalently, at least one of
the numbers $b$, $b'$ is even), then both $\frac 12bb'$ and $2\Delta$
are integers.
In both cases, the numbers $2F(\eb_{12},\eb_{21})$ and $2F(\eb_{11},\eb_{22})$
are integers.
\end{proof}

We can now prove the part of Theorem A dealing with $\delta(F)$.

\begin{prop}
  \label{p:delta}
  We have $\delta(F)=\delta(f)\delta(f')$.
\end{prop}

\begin{proof}
  Set $m=\delta(f)$, $m'=\delta(f')$ and $M=mm'$.
  From the identity $F(\xb\circ\xb')=f(\xb)f'(\xb')$, it follows that the
  greatest common divisor of the values of $F$ at all products $\xb\circ\xb'$
  is equal to $M$.
  It remains to prove that the values $F(\zb)$ at all other points $\zb\in\Z^2$
  are divisible by $M$.
  This is equivalent to proving that the quadratic form $F/M$ is integer.
  To this end, note that $F/M$ is composed of the integer quadratic forms $f/m$
  and $f'/m'$ by means of the same composition law $\circ$.
  The result now follows from Proposition \ref{p:int}.
\end{proof}

We now prove the part of Theorem A dealing with $\sigma(F)$.

\begin{prop}
  \label{p:sigma}
  We have $\sigma(F)=\min(\sigma(f),\sigma(f'))$.
\end{prop}

\begin{proof}
Suppose first that $\sigma(f)=\sigma(f')=1$.
Then, dividing the forms $f$, $f'$ by suitable powers of 2 if necessary,
we may assume that $b$ and $b'$ are odd.
By Proposition \ref{p:Del}, the number $4\Delta=\pm\sqrt{dd'}$ is also odd.
We need to prove in this case that not all inner products with respect
to $F$ are integer.
From formulas (\ref{e:Del1}) and (\ref{e:Del2}), we obtain that
$$
F(\eb_{12},\eb_{21})-F(\eb_{11},\eb_{22})=2\Delta.
$$
Since $2\Delta$ is not integer, one of the numbers $F(\eb_{12},\eb_{21})$ or
$F(\eb_{11},\eb_{22})$ is also not integer.

Suppose now that $\sigma(f)=2$, $\sigma(f')=1$
(the case $\sigma(f')=2$, $\sigma(f)=1$ is similar).
Dividing the forms $f$ and $f'$ by suitable powers of 2 and interchanging
$x$ with $y$ if necessary, we may assume that $a$ and $b'$ are odd but $b$ is even.
In this case $F(\eb_{11},\eb_{12})=\frac 12ab'$ is half-integer, hence $\sigma(F)=1$.

Finally, suppose that $\sigma(f)=\sigma(f')=2$.
We may assume that $a$ and $a'$ are odd, while $b$ and $b'$ are even.
The value $F(\eb_{11})=aa'$ is odd.
Thus it suffices to show that all inner products with respect to $F$
are integer.
Indeed, the four numbers
$$
F(\eb_{i1},\eb_{i2})=\frac 12 f(\eb_i)b',\quad F(\eb_{1j},\eb_{2j})=\frac 12 bf'(\eb_j),
\quad i,j=1,2,
$$
are integer since $b$ and $b'$ are even.
Note also that $\Delta$ is integer, hence the numbers
$$
F(\eb_{12},\eb_{21})=\frac 14bb'-\Delta,\quad F(\eb_{11},\eb_{22})=\frac 14bb'+2\Delta
$$
are integer as well.
\end{proof}

\section{Composition law}
In this section, we study the properties of the bilinear map
$\circ:\Z^2\times\Z^2\to\Z^2$.
Recall that, by our assumption, the ``products'' $\xb\circ\yb$ span $\Z^2$.

\begin{prop}
  \label{p:gcd}
  Let $e$ be a positive integer.
  If all numbers $\det(\xb\circ)$, where $\xb\in\Z^2$, and all numbers
  $\det(\circ\yb)$, where $\yb\in\Z^2$, are divisible by $e$, then $e=1$.
\end{prop}

\begin{proof}
  By the assumption, all numbers of the form $\det(\xb\circ\yb,\xb\circ\yb')$
  and all numbers of the form $\det(\xb\circ\yb,\xb'\circ\yb)$ are divisible by $e$.
  Consider the determinant
  $$
  \det((\xb+\xb')\circ\yb,(\xb+\xb')\circ\yb'),
  $$
  which is divisible by $e$ by our assumption.
  Open the parentheses in this determinant, using the bilinearity.
  Dropping the terms that are known to be divisible by $e$, we obtain that
  \begin{equation}
  \label{e:plus}
  \det(\xb\circ\yb,\xb'\circ\yb')+\det(\xb'\circ\yb,\xb\circ\yb')\equiv 0\pmod e.
  \end{equation}
  Similarly, if we start with the determinant
  $$
  \det(\xb\circ(\yb+\yb'),\xb'\circ(\yb+\yb')),
  $$
  we obtain that
  \begin{equation}
  \label{e:minus}
  \det(\xb\circ\yb,\xb'\circ\yb')-\det(\xb'\circ\yb,\xb\circ\yb')\equiv 0\pmod e.
  \end{equation}
  It follows from equations (\ref{e:plus}) and (\ref{e:minus}) that the numbers
  $\frac 2e\det(\xb\circ\yb,\xb'\circ\yb')$ are integer for all $\xb$, $\yb$, $\xb'$,
  $\yb'\in\Z^2$.
  Therefore, the numbers $\frac 2e\det(\zb,\zb')$ are integer for all
  $\zb$, $\zb'\in\Z^2$.
  Since the greatest common divisor of all values $\det(\zb,\zb')$ is equal to one,
  it follows that $e=1$ or 2.

  Suppose that $e=2$.
  We will write $\F_2$ for the field with 2 elements, and $\ol\eb_{ij}\in\F_2^2$
  for the mod 2 reduction of the vector $\eb_{ij}\in\Z^2$.
  Since we have
  $$
  \det(\ol\eb_{ii},\ol\eb_{ii})=\det(\ol\eb_{ii},\ol\eb_{ij})=
  \det(\ol\eb_{ij},\ol\eb_{ij})=0,
  $$
  and the form $\det$ is nonzero, the vectors $\ol\eb_{ii}$ and $\ol\eb_{ij}$
  must be linearly dependent for every pair $i,j=1,2$
  (otherwise this pair of vectors span $\F_2^2$, and we can easily conclude that
  $\det(\ol\zb,\ol\zb')=0$ for all $\ol\zb$, $\ol\zb'\in\F_2^2$).
  Similarly, vectors $\ol\eb_{ij}$ and $\ol\eb_{jj}$ are linearly dependent.
  But then all four vectors $\ol\eb_{ij}$ are linearly dependent.
  A contradiction with the fact that these four vectors span $\F_2^2$.
  The contradiction shows that $e=1$.
\end{proof}

The following proposition concludes the proof of Theorem A:

\begin{prop}
  \label{p:theta}
  We have $\theta(F)=\gcd(\theta(f),\theta(f'))$.
\end{prop}

\begin{proof}
Each side of equation (\ref{e:circy}) represents a set of numbers parameterized by $\xb'$.
Since the two sets are equal, they have the same greatest common divisor.
Thus we obtain that $d\delta(f')^2/D$ is the greatest
common divisor of all numbers $\det(\circ\xb')$, where $\xb'$ runs through $\Z^2$.
Similarly, we obtain from equation (\ref{e:xcirc}) that $d'\delta(f)^2/D$
is the greatest common divisor of all numbers $\det(\xb\circ)$, where $\xb$ runs
through $\Z^2$.

Proposition \ref{p:gcd} implies that
$$
\gcd\left(d\delta(f')^2,d'\delta(f)^2\right)=D.
$$
Dividing both parts of this equation by $\delta(F)^2=\delta(f)^2\delta(f')^2$,
we obtain that $\theta(F)=\gcd(\theta(f),\theta(f'))$, as desired.
\end{proof}

Finally, we can conclude the proof of Theorem B.

\begin{prop}
\label{p:exist}
  Consider two binary integer forms $f$ and $f'$ such that
  $d(f)/d(f')$ is a square of a rational number.
  Then there exists a form $F$ and a composition law $\circ$ such that
  $F$ is composed of $f$ and $f'$ by means of $\circ$.
\end{prop}

Suppose that $f=(a,b,c)$ and $f'=(a',b',c')$.
As before, we set $d=d(f)$ and $d'=d(f')$.
Let $r$ be a rational number such that $d'=dr^2$.

We start with several remarks.
Note that $f$ can be reduced to the form
$f(x,y)=a(X^2-d Y^2)$,
where $X$ and $Y$ are linear combinations of $x$ and $y$
with rational coefficients.
Indeed, we have
$$
f(x,y)=ax^2+bxy+cy^2=
a\left(\left(x+\frac{b}{2a}y\right)^2-d\left(\frac{y}{2a}\right)^2\right).
$$
Similarly, the form $f'$ can be reduced to the form
$a'({X'}^2-dr^2{Y'}^2)$, where ${X'}$ and
${Y'}$ are
linear combinations of $x$ and $y$ with rational coefficients.
Since $r^2{Y'}^2=(r{Y'})^2$, the form $f'$ can also be reduced
to the form $a'(1,0,-d)$ by a linear substitution with rational coefficients.

Let $N$ denote the form $(1,0,-d)$.
We will think of $N$ as a quadratic form defined on $\Q^2$.
Then, by the above, there exist $\Q$-linear maps
$\varphi:\Q^2\to\Q^2$ and $\varphi':\Q^2\to\Q^2$ such that
\begin{equation}
\label{e:red}
aN(\varphi(x,y))=f(x,y),\quad a'N(\varphi'(x,y))=f'(x,y).
\end{equation}

\begin{proof}[Proof of Proposition \ref{p:exist}]
  Consider the algebra $\A=\R[t]/(t^2-d)$ (the quotient of the polynomial
  algebra $\R[t]$ by the principal ideal generated by $t^2-d$).
  We will write $\eps$ for the class of the polynomial $t$ in $\A$.
  Thus $\A$ is a vector space over $\R$ with basis $1$, $\eps$, and we
  have $\eps^2=d$.
  (If $d$ is negative, then $\A$ is isomorphic to $C$; if $d$ is positive,
  then $\A$ is isomorphic to the algebra of hyperbolic numbers.)
  For every element $u=x+y\eps$, we set $\ol u=x-y\eps$.
  We will write $N(u)$ for $u\ol u$.
  Since $N(x+y\eps)=x^2-dy^2$, this is consistent with our previous notation $N=(1,0,-d)$.

  Consider the $\Q$-linear maps $\varphi$, $\varphi':\Q^2\to\A$, for which
  formula (\ref{e:red}) holds.
  Define the lattices $L=\sqrt{a}\varphi(\Z^2)$, $L'=\sqrt{a'}\varphi'(\Z^2)$.
  The product $LL'$ is defined as the additive subgroup of $\A$ spanned
  by the products $uu'$, where $u$ runs through $L$ and $u'$ runs through $L'$.
  Note that this construction is similar to that given in Remark \ref{r:rings}.
  The additive group $LL'$ is a finitely generated subgroup of the additive group
  $\A=\R^2$.
  It follows from the classification of finitely generated Abelian groups
  that $LL'$ is isomorphic to $\Z^2$.
  Let $s:LL'\to\Z^2$ be an isomorphism.

  We can now define a bilinear map $\circ:\Z^2\times\Z^2\to\Z^2$ by
  the formula $\xb\circ\xb'=s(\sqrt{aa'}\varphi(\xb)\varphi'(\xb'))$.
  The product of $\varphi(\xb)$ and $\varphi'(\xb')$ is taken in the algebra $\A$.
  We also define a quadratic form $F$ on $\Z^2$ by the formula $F(\zb)=N(s^{-1}(\zb))$.
  The property $N(uv)=N(u)N(v)$ now implies that $F$ is composed of $f$ and $f'$
  by means of $\circ$.
\end{proof}

\end{document}